%% file: Simpson2026finite.tex
\documentclass{IEEEtran}

\usepackage{amsmath, amsfonts, amssymb, amsthm, mathtools}
\usepackage{cite, graphicx}
\usepackage{textcomp}
\usepackage[hidelinks]{hyperref}
\usepackage{xcolor}

\def\BibTeX{{\rm B\kern-.05em{\sc i\kern-.025em b}\kern-.08em
      T\kern-.1667em\lower.7ex\hbox{E}\kern-.125emX}}

\newtheorem{thm}{Theorem}
\newtheorem{assum}{Assumption}
\newtheorem{lem}{Lemma}

\newtheorem{cor}{Corollary}
\newtheorem{defn}{Definition}

\begin{document}
% \renewcommand{\coloneqq}{=}
% \setlength{\parindent}{0pt}

\input{macros}

\title{\LARGE \bf
Finite-sample bounds for multi-output system identification
}

\author{
   {L\'eo Simpson$^{1, 2}$, Katrin Baumg\"{a}rtner$^1$, Johannes K\"{o}hler$^3$, and Moritz Diehl$^{1,2}$}
   \thanks{
         This research was supported by DFG via projects 560056112 (robust MPC) and 535860958 (ALeSCo), and by the EU via ELO-X 953348.
   }
   \thanks{
      $^1$ L\'eo Simpson, Katrin Baumg\"{a}rtner, and Moritz Diehl are with the Department of Microsystems Engineering (IMTEK), University of Freiburg, 79110 Freiburg, Germany (e-mails: \mbox{firstname.familyname@imtek.uni-freiburg.de}).
   }
   \thanks{
      $^2$ L\'eo Simpson and Moritz Diehl are also with the Department of Mathematics, at the same university.
   }
   \thanks{
      $^3$ Johannes K\"{o}hler is with the Department of Mechanical Engineering, Imperial College London, London, UK (e-mail: \mbox{j.kohler@imperial.ac.uk}).
   }
}

\maketitle

% \coloredrevisiontrue   % revisions colored, comments visible
\coloredrevisionfalse  % revisions normal, comments hidden

\extendedversiontrue   % this is for the arXiv version (longer than 8pages)
% \extendedversionfalse  % this is for the TAC version (8pages)

\input{abstract}

\begin{IEEEkeywords}
   Statistical learning, system identification, closed-loop identification, data-driven control, learning for MPC.
\end{IEEEkeywords}

\input{content}

% \clearpage
\bibliographystyle{IEEEtran}
\bibliography{biblio}

\appendix
\renewcommand{\theequation}{A.\arabic{equation}}
% reset the counter
\setcounter{equation}{0}

\input{appendix}

\end{document}

%% file: macros.tex
% Math letters
\newcommand{\cl}[1]{{\mathcal #1}}
\newcommand{\mr}[1]{{\mathrm{#1}}}
\newcommand{\R}{{\mathbb{R}}}
\newcommand{\E}{{\mathbb{E}}}
\newcommand{\N}{{\mathbb{N}}}
\renewcommand{\P}{{\mathbb{P}}}
\newcommand{\Id}{{I}}
\renewcommand{\O}{{\cl{O}}}

% Commands for inequality relations
\newcommand{\GEQ}{{\, \succcurlyeq \, }}
\newcommand{\LEQ}{{\, \preccurlyeq \, }}
\newcommand{\GEQS}{{\, \succ \, }}
\newcommand{\LEQS}{{\, \prec \, }}

% Maths operations
\newcommand{\of}[1]{\left(#1\right)}
\newcommand{\curlyof}[1]{\left\{#1\right\}}
\newcommand{\bigof}[1]{{\big(#1\big)}}
\newcommand{\biggof}[1]{{\bigg(#1\bigg)}}
\newcommand{\ofc}[1]{{\left[#1\right]}}
\newcommand{\Trace}[1]{\mr{Tr}\!\of{#1}}
\newcommand{\abs}[1]{\left\lvert#1\right\rvert}
\newcommand{\norm}[1]{\left\lVert#1\right\rVert}
\newcommand{\smallnorm}[1]{\lVert#1\rVert}
\newcommand{\fnorm}[1]{ \norm{#1}_{\mr{Fr}}}
\newcommand{\onorm}[1]{\norm{#1}_{\mr{op}} }
\newcommand{\smallonorm}[1]{\smallnorm{#1}_{\mr{op}} }
\newcommand{\mylimit}[1]{\xrightarrow[#1\to+\infty]{}}
\newcommand{\setdef}[2]{\left\{ #1 \; \text{s.t.} \; #2 \right\}}
\renewcommand{\det}[1]{\abs{#1}}
\newcommand{\bigdet}[1]{\big\lvert#1\big\rvert}
\renewcommand{\exp}[1]{e^{#1}} % or \renewcommand{\exp}[1]{\exp\of{#1}}
\renewcommand{\vec}[1]{\mr{vec}\!\of{#1}}
\newcommand{\half}{{\frac{1}{2}}}
\newcommand{\logdet}[1]{\log \det{#1}}

% Probability
\newcommand{\Eof}[1]{\E \left[#1\right]}
\newcommand{\sEof}[1]{\E [#1]}
\newcommand{\Covof}[1]{\mr{Cov}\left[#1\right]}
\newcommand{\Pof}[1]{\P\left[#1\right]}
\newcommand{\Gauss}[2]{\cl{N} \of{ #1, \, #2 }  }
\newcommand{\Gausstiny}[2]{\cl{N}(#1, \, #2)  }

% Commands to make things small
\newcommand{\splus}{\!+\!} % small +
\newcommand{\sminus}{\!-\!} % small -
\newcommand{\sequal}{\!=\!} % small =

% My variables
\newcommand{\ny}{{m}}
\newcommand{\ntheta}{{n}}
\newcommand{\nphi}{{n_{\varphi}}}
\newcommand{\nx}{{n_x}}
\newcommand{\nuu}{{n_u}}

\newcommand{\cw}{c_w}
\newcommand{\ctheta}{c_\theta}
\newcommand{\cM}{{c_M}}
\newcommand{\Set}{{\cl{C}}}
\newcommand{\tmax}{{T}}
\newcommand{\Pprior}{{P}}
\newcommand{\muprior}{\mu} % \theta^{\text{prior}}
\newcommand{\shift}{{z}}
\newcommand{\I}{{\N_{\leq \tmax}}}
\newcommand{\Iminus}{{\N_{\leq \tmax-1}}}
\newcommand{\event}{\cl{E}}

\newcommand{\IMAGES}{images}

% Editing
\definecolor{colorrevisions}{HTML}{2368b8}
% Toggle colored revisions:
\newif\ifcoloredrevision
\newcommand{\revision}[1]{%
  \ifcoloredrevision {\color{colorrevisions}#1}%
  \else {#1}%
\fi}

\newif\ifextendedversion

%% file: abstract.tex
\begin{abstract}
    This paper presents uniform-in-time error bounds for regularized linear regression with vector-valued outputs and conditionally zero-mean subgaussian noise.
    By leveraging self-normalized martingale arguments, we obtain bounds that apply directly to multi-output regression.
    \revision{These bounds are tighter and hold in more general settings than commonly used methods in finite-sample system identification.
    Furthermore, because of a novel treatment of the regularization bias, they are tighter even in the scalar-output case.}
    The mild assumptions we use allow for unknown dependencies between regressors and past noise terms, typically induced by system dynamics or feedback mechanisms.
    Therefore, these novel bounds can be applied to many affine-in-parameter system identification problems, including the identification of a linear time-invariant system from full-state measurements.
    These results \revision{enable less conservative uncertainty quantification in learning-based control, which may significantly improve performance in safety-critical applications.}
\end{abstract}

%% file: content.tex
\section{Introduction}\label{section-introduction}

System identification is widely employed~\cite{ljung1999system}, either on a fixed training dataset or in adaptive settings, where parameter estimates are used online for control or decision-making~\cite {hou2013model}.
In both cases, it is often essential to quantify estimation uncertainty to enforce safety constraints or to provide performance guarantees.
In contrast to deterministic approaches~\cite{kosut1992set,lahr2026optimal}, we consider the stochastic setting, where the noise is modeled as random vectors.
For this setting, classical asymptotic results describe behavior of the estimates as the sample size grows~\cite{Ljung1978asymptotic}, but provide limited guidance for finite data.
Motivated by safety-critical learning-based decision-making, many works have developed non-asymptotic theory, where uncertainty is quantified via confidence sets~\cite{care2017finite, campi2002finite,abbasi2011improved}.
These confidence sets are derived in frequentist terms: the stochasticity solely models the data-generation process, not the lack of knowledge about the system.
Notably, such bounds have been employed
for Optimism in Face of Uncertainty (OFU) policies~\cite{abbasi2011improved,baumgartner2026gray,chowdhury2017kernelized,abbasi2011regret};
for safety in learning-based Model Predictive Control (MPC)~\cite{koller2018learning,lew2022safe} and in Bayesian optimization~\cite{sui2015safe};
and for optimal experiment design~\cite{wagenmaker2020active}.

In such applications, a central challenge is that, due to the system dynamics or the feedback mechanisms, the regressors can be correlated with past noise, which prevents the use of standard linear regression results, \revision{or methods such as the \textit{sign-perturbed sums} method, which is restricted to regressors independent of the noise.
Some variations of this method overcome this limitation, but they are conservative and restricted to specific models; see \cite{care2017finite} for an overview.
Another major practical disadvantage of these variants is that the resulting confidence sets do not always contain the least-squares estimates.
An earlier work constructs finite-sample bounds for identification of transfer functions by adapting the proof of classical asymptotic results~\cite{campi2002finite}.
However, these bounds are quite conservative, as they scale exponentially in the problem dimensions.
Importantly, all methods discussed so far only provide pointwise-in-time probability guarantees, which makes them inadequate for online learning.
}

\revision{
    In parallel, Self-Normalized Martingale (SNM) bounds~\cite{DeLaPena2009self} have been successfully applied to online learning to overcome these issues~\cite{abbasi2011improved}.
    These bounds were derived in the context of learning an objective function, so it was natural to restrict to scalar outputs.
    More recently, \cite{chowdhury2021noregret} generalized these results to vector-valued outputs in the context of multi-task Bayesian optimization.
    Many subsequent works have built upon the scalar result~\cite{abbasi2011regret,sui2015safe,koller2018learning,chowdhury2017kernelized,wagenmaker2020active,lew2022safe,whitehouse2026time, martinez2026vector,aolaritei2025stopping}, but the vector-valued generalization seems to have drawn little attention.
    This is notably the case for the recent works~\cite{tsiamis2023statistical,ziemann2023tutorial,matni2019tutorial} that brought SNM arguments to the system identification setting.
    These arguments have to be combined with additional arguments to handle multi-output systems.
    Such additional steps make the bounds more conservative and restrict their applicability to structure-agnostic identification of linear systems.
}
\paragraph*{Contribution}
\revision{
    By applying classical SNM bounds directly to multi-output regression, together with a special treatment of the regularization bias, this paper presents finite-sample bounds for system identification that overcome the limitations mentioned above.
    Compared to~\cite{abbasi2011improved,chowdhury2021noregret}, we provide these bounds for a wider range of settings: regularized or ordinary least-squares, pointwise or uniform in time probability guarantees, structured or unstructured linear models.
}

\paragraph*{Outline}
Section~\ref{section-problem} introduces the regression problem, its underlying assumptions, and the finite-sample bounds sought in this paper.
In Section~\ref{section-main-results}, we present the main results: a novel SNM bound and resulting confidence sets.
These results are discussed in Section~\ref{section-discussion}, with different use cases and comparisons with the state of the art.
Finally, Section~\ref{section-numerical-experiments} illustrates the presented results through numerical examples, and Section~\ref{section-conclusion} presents conclusions and an outlook.

\paragraph*{Notation}
Throughout the paper, we use the notation ${\I \coloneqq \curlyof{0, \dots, \tmax}}$ for \revision{$\tmax \in \N$}.
The capitalized letters all denote matrices, except for $\tmax$.
\revision{We denote by $\Id_n$ the identity matrix of size $n \in \N$.}
The determinant and the spectral radius of a square matrix $A$ are denoted by $\det{A}$ and $\rho\of{A}$, respectively.
We use the Euclidean norm ${\norm{x}^2 \coloneqq x^\top x}$ and the weighted (semi)norm ${\norm{x}_P^2 \coloneqq x^\top P x}$ for positive (semi)definite weighting matrices $P$.
For matrices, we use the Frobenius norm ${\fnorm{A}^2 \coloneqq \Trace{A^\top A}}$ and the operator norm ${\onorm{A} \coloneqq \sup_{\norm{x}\le1} \norm{A x}}$.
We use the Kronecker product $A \otimes B$ and denote by $\vec{A}$ the vector obtained by stacking the rows of $A$.
We denote by $\Gausstiny{\mu}{\Sigma}$ the Gaussian distribution with mean $\mu$ and covariance matrix $\Sigma$.

\section{The setup}\label{section-problem}

\subsection{The linear regression model}\label{subsection-model}
Consider linear regression tasks in the general form
\begin{equation}\label{eq:model}
    y_{t+1} = M_t \theta^\star + w_t, \qquad t \in \Iminus,
\end{equation}
where $\theta^\star \in \R^{\ntheta}$ is an unknown parameter vector to be estimated, $w_t \in \R^{\ny}$ are unmeasured noise vectors, $y_t \in \R^{\ny}$ are measured outputs, and $M_t \in \R^{\ny \times \ntheta}$ are measured regressors.
The dataset at time step $t$ is composed of
\begin{equation}\label{eq:data-definition}
    \cl{D}_t \coloneqq \curlyof{M_0, y_1, M_1, \dots, y_t, M_t}.
\end{equation}
At each time $t \in \I$, we construct an estimator $\hat{\theta}_t$ of the true parameter $\theta^\star$ based on the available data $\cl{D}_t$.
This is done by solving a regularized linear least squares problem:
\begin{equation}\label{eq:opti-estimates}
        \min_{\theta \in \R^{\ntheta}} \, \norm{\theta-\muprior}^2_{\Pprior} + \sum_{k=0}^{t-1} \norm{y_{k+1} - M_k \theta}^2_W,
\end{equation}
where $\muprior \in \R^{\ntheta}$ and $\Pprior \GEQ 0$ represents some prior belief, and $W \GEQS 0$ is a weighting matrix.
For $\Pprior = 0$, we recover the Ordinary Least Squares (OLS) estimator.
We define the following variables
\begin{equation}\label{eq:def-s-and-V}
        V_t \coloneqq \sum_{k=0}^{t-1} M_k^\top W M_k, \qquad
        s_t \coloneqq \sum_{k=0}^{t-1} M_k^\top W w_k.
\end{equation}
Whenever $P+V_t \GEQS 0$, \eqref{eq:opti-estimates} has a unique solution given by
\begin{equation}\label{eq:estimates-expression}
\begin{aligned}
    \hat{\theta}_t &= \of{\Pprior + V_t }^{-1} \of{\Pprior \muprior + \sum_{k=0}^{t-1} M_k^\top W y_{k+1}}, \\
    &= \theta^\star + (\Pprior + V_t)^{-1} \biggof{\Pprior (\muprior - \theta^\star) + s_t}.
\end{aligned}
\end{equation}

\subsection{Finite-sample bounds}\label{subsection:confidence-sets}
The goal is to determine a sequence of data-dependent variables $\beta_t \ge 0$ such that the following bound holds:
\begin{equation}\label{eq:confidence-set}
    \Pof{\norm{\theta^\star - \hat{\theta}_t}_{\Pprior + V_t} \le \beta_t, \quad \forall t \in \I} \ge 1 - \delta,
\end{equation}
where  $\delta \in (0,1)$ is a user-defined \revision{failure probability}.
When ${\Pprior+V_t}$ is singular, we always choose $\beta_t = +\infty$ since the solution of \eqref{eq:opti-estimates} is not unique.
\revision{
    Otherwise, when ${\Pprior+V_t \GEQS 0}$, \eqref{eq:confidence-set} also implies a bound on the unweighted norm $\smallnorm{\theta^\star - \hat{\theta}_t}$.
    In this paper, we focus on the weighted norm instead because exploiting the information about the direction of the uncertainty usually yields tighter confidence sets.
}
Notably, \eqref{eq:confidence-set} conveniently implies the following output bound:
\begin{align}\label{eq:output-bound}
    \P \bigg[
        \norm{M \theta^\star - M \hat{\theta}_t } \le \beta_t 
        \overbrace{\onorm{M \of{\Pprior + V_t}^{-1/2}}}^{\eqqcolon \sigma_t(M)}
        ,&
        \\ \nonumber & \hspace{-30mm}
        \forall t \in \I, \forall M \in \R^{\ny \times \ntheta}
    \bigg] \ge 1 - \delta.
\end{align}
Importantly, this probability bound is uniform in time, which is particularly useful in online learning settings.
\revision{We limit the analysis to a finite horizon length $\tmax$.
This is not restrictive as one can choose $\tmax$ arbitrarily large without affecting the confidence sets in~\eqref{eq:confidence-set}.
}

The probability in~\eqref{eq:confidence-set} or in~\eqref{eq:output-bound} is taken over the data-generation process only; the true parameter $\theta^\star$ is deterministic while being unknown.
This is the frequentist perspective, which has a direct empirical interpretation: $\delta$ upper-bounds the frequency of invalid confidence sets when repeating the data-generation and identification process many times.
Interestingly, this holds true even when repeating the process for different values of $\theta^\star$.

This is different from the Bayesian approach, where probabilities quantify one's belief about the true parameter $\theta^\star$.
Deriving finite-sample bounds in the Bayesian sense is usually intractable, except in special cases such as Gaussian (and exactly known) distributions for the prior and the noise.
In contrast, we derive frequentist finite-sample bounds under mild assumptions.

\subsection{Assumptions}\label{subsection:assumptions}

To derive such bounds, we make the following standard (cf.~\cite{abbasi2011improved,ziemann2023tutorial}) assumption on the noise.
\begin{defn}[Subgaussian variables]\label{defn-subgaussian}
    A zero-mean random vector $w \in \R^{\ny}$ is subgaussian with variance proxy $Q \GEQ 0$ if and only if $\Eof{\exp{\nu^\top w}} \le \exp{\half \norm{\nu}^2_Q}$ for all $\nu \in \R^{\ny}$.
\end{defn}
\begin{assum}[Conditional subgaussian noise]\label{assum:subgaussian}
    For a known constant $\cw > 0$, the noise $w_t$ is conditionally\footnotemark zero-mean and subgaussian with variance proxy $\cw^2 W^{-1}$ for all $t \in \Iminus$:
    \begin{equation}\label{eq:assum:subgaussian}
        \forall \nu_t \in \R^{\ny}, \quad
        \Eof{\exp{\nu_t^\top w_t} \mid \cl{D}_t} \le \exp{\half \cw^2 \norm{\nu_t}^2_{W^{-1}}}.
    \end{equation}
\end{assum}
\footnotetext{\revision{
    Instead of conditioning on the datasets $\cl{D}_0, \dots, \cl{D}_T$, an alternative formalism is to assume that there exists a filtration $\cl{F}_0, \dots, \cl{F}_T$ such that $\cl{D}_t$ is $\cl{F}_t\textup{-measurable}$ for all $t \in \I$, and that \eqref{eq:assum:subgaussian} holds conditionally on $\cl{F}_t$ instead.
    These two approaches are equivalent, so we chose the former to keep the exposition more concise and intuitive.
}}

Note that $\Eof{w_t \mid \cl{D}_t} = 0$ is implied by~\eqref{eq:assum:subgaussian}.
Importantly, Assumption~\ref{assum:subgaussian} includes the case of Gaussian noise \revision{${w_t \mid \cl{D}_t \sim \Gausstiny{0}{Q_t}}$}, as long as $\cw^2 \ge \rho(Q_t W)$ for all ${t \in \Iminus}$.
It also includes \revision{conditionally} zero-mean distributions for which $\smallnorm{w_t}_{W} \le \cw$ almost surely.
This comes from Hoeffding's inequality~\cite[Theorem 2.8]{boucheron2004concentration}.
Crucially, Assumption~\ref{assum:subgaussian} allows correlations between the regressor $M_t$ and the past noise terms $w_k$ for $k < t$, which is essential for the application of the present results in system identification.

Next, we state a
\revision{standard (cf. \cite{abbasi2011improved,baumgartner2026gray,chowdhury2017kernelized,abbasi2011regret,koller2018learning})}
assumption on the true parameter $\theta^\star$.
\begin{assum}[Bound on the true parameter]\label{assum:bounded-prior}
    A constant $\ctheta \ge 0$ is known such that the true parameter $\theta^\star$ satisfies:
    \begin{equation}\label{eq:bounded-prior}
        \norm{\theta^\star - \muprior}_{\Pprior} \le \ctheta.
    \end{equation}
\end{assum}
\revision{
    It is often possible to derive a loose bound in this form from prior knowledge about the system.
    When this is not the case, one should use OLS, i.e., $\Pprior = 0$, to make this assumption trivially satisfied with $\ctheta = 0$.
}

\section{Main results}\label{section-main-results}
Using the analytical formula~\eqref{eq:estimates-expression} for the estimates, we can state the finite-sample bound~\eqref{eq:confidence-set} equivalently as
\begin{equation}\label{eq:confidence-set-reformulated}
    \Pof{ \norm{\Pprior (\muprior - \theta^\star) + s_t}_{(\Pprior+V_t)^{-1}} \le \beta_t, \quad \forall t \in \I} \ge 1 - \delta.
\end{equation}
This bound is called a Self-Normalized Martingale (SNM) bound because $s_t$ is a martingale and the inverse of $\Pprior+V_t$ normalizes its value.
To derive SNM bounds, the main challenges are the regressor-noise correlations and the uniform-in-time requirement.
The seminal result~\cite[Theorem 1]{abbasi2011improved} overcame these difficulties, but only for scalar outputs (i.e., $\ny=1$).
\revision{This was generalized to vector-valued outputs in~\cite{chowdhury2021noregret}.}
In the theorem below, \revision{we prove a similar generalization} and incorporate an offset $\shift \in \R^{\ntheta}$ in the martingale, which will be used later to treat the regularization bias $\Pprior (\muprior - \theta^\star)$ in~\eqref{eq:confidence-set-reformulated}.

\begin{thm}[Self-normalized martingale bound]\label{thm:main-result}
    Let Assumption~\ref{assum:subgaussian} hold.
    Let $\shift \in \R^{\ntheta}$ be a vector and $\bar{P} \GEQS 0$ be a positive definite matrix.
    Then, the following inequality holds with uniform-in-time probability at least $1 - \delta$:
    \begin{align}\label{eq:main-result}
        \norm{\shift + s_t}_{(\bar{P}+V_t)^{-1}}^2 \le& \norm{\shift}_{\bar{P}^{-1}}^2 +
        \\ \nonumber & \quad
        \cw^2 \of{\logdet{\Id_\ntheta +\bar{P}^{-1} V_t} + 2\log\of{1/\delta}},
    \end{align}
    where $V_t \in \R^{\ntheta \times \ntheta}$ and $s_t \in \R^{\ntheta}$ are defined in~\eqref{eq:def-s-and-V}.
\end{thm}

\subsection{Proof of Theorem~\ref{thm:main-result}}
The proof is conceptually similar to~\cite[Theorem 1]{abbasi2011improved}.
It relies on an inequality for subgaussian variables (Lemma~\ref{lemma:subgaussian}) and a uniform-in-time probability bound for supermartingales (Lemma~\ref{lemma:supermartingale}).

\begin{lem}[A concentration inequality]\label{lemma:subgaussian}
    Let $v \in \R^{\ntheta}$ be a zero-mean subgaussian variable with variance proxy $R \GEQ 0$.
    Then, for any matrix $H \GEQS 0$, and any vector $x \in \R^{\ntheta}$, the following holds:
    \begin{equation}\label{eq:subgaussian-concentration}
        \frac{1}{\sqrt{\det{H+R}}} \Eof{e^{\half\norm{x + v}_{\of{H + R}^{-1} }^2} }
	    \leq \frac{1}{\sqrt{\det{H}}} e^{\half\norm{x}_{H^{-1}}^2}.
    \end{equation}
\end{lem}
The special case $x=0$ was presented in~\cite[Theorem 14.7]{DeLaPena2009self}.
We provide the general proof in the appendix.

To provide intuition about the connection between Lemma~\ref{lemma:subgaussian} and the desired bound~\eqref{eq:main-result}, remark that applying Markov's inequality to the left-hand side of~\eqref{eq:subgaussian-concentration} and taking the logarithm of both sides yields, with probability at least $1 - \delta$:
\begin{equation}\label{eq:general-subgaussian-bound}
    \norm{x + v}_{\of{H + R}^{-1}}^2 \! \!
    \leq \norm{x}_{H^{-1}}^2 + \logdet{\Id_\ntheta +H^{-1}R} + 2\log\of{1/\delta}.
\end{equation}
If all $M_k$ were deterministic (or independent of all $w_{k'}$), $s_t$ would be subgaussian with proxy $V_t$, so one could directly apply~\eqref{eq:general-subgaussian-bound} to prove a pointwise version of Theorem~\ref{thm:main-result}.
Unfortunately, $M_t$ is a random matrix, so we need to tackle the problem differently: Lemma~\ref{lemma:subgaussian} needs to be applied recursively, conditionally on the data $\cl{D}_t$.

Next, we introduce Lemma~\ref{lemma:supermartingale} below to address the uniform-in-time aspect of the problem.
\begin{lem}[Supermartingale bound]\label{lemma:supermartingale}
    Let $q_0, \dots, q_{\tmax}$ be a non-negative supermartingale, i.e., a sequence of non-negative random variables such that $q_t$ is a function of $\cl{D}_t$ only, and satisfies $q_t \ge \Eof{q_{t+1} \mid \cl{D}_t}\ge 0$.
    Then, the following uniform-in-time probability bound holds true for any $\delta \in (0, 1)$:
    \begin{equation}\label{eq:supermartingale-inequality}
        \Pof{ q_t \le \frac{\Eof{q_0}}{\delta}, \quad \forall t \in \I} \ge 1 - \delta.
    \end{equation}
\end{lem}
This lemma was proven, for example, in~\cite[Theorem 7.3.1]{shiryaev2019probability} using stopping time theory.
\ifextendedversion ~
\else
{
\footnotemark
\footnotetext{\revision{
We also provide a proof using only elementary tools in the appendix of the extended version \underline{\url{https://arxiv.org/pdf/2603.19073}}.
}}
}
\fi

Now we have all the ingredients to prove Theorem~\ref{thm:main-result}.
\begin{proof}[Proof of Theorem~\ref{thm:main-result}]
    We consider $\cw=1$ without loss of generality since we can easily rescale the variables: $w_t \leftarrow w_t /\cw$, $s_t \leftarrow s_t / \cw$, and $\shift \leftarrow \shift /\cw$.
    Therefore, by Assumption~\ref{assum:subgaussian}, $w_t$ is conditionally subgaussian with variance proxy $W^{-1}$.
    By linear transformation of subgaussian variables, $v_t \coloneqq M_t^\top W w_t$ is also conditionally subgaussian with proxy $R_t \coloneqq M_t^\top W M_t$.
    Now, define $H_t \coloneqq \bar{P} + V_t \GEQS 0$.
    Note that $s_{t+1} = s_t + v_t$ and that $H_{t+1} = H_t + R_t$.
    Therefore, applying Lemma~\ref{lemma:subgaussian} conditionally on $\cl{D}_t$, with $H=H_t$, $x = \shift + s_t$, $R=R_t$, and $v = v_t$, we find:
    \begin{equation}\label{eq:proof:main-1}
        \E \bigg[
            \underbrace{
                \frac{1}{\sqrt{\det{H_{t+1}}}}e^{\half\norm{\shift + s_{t+1}}_{H_{t+1}^{-1} }^2}
            }_{\coloneqq q_{t+1}}
            \Bigg| \cl{D}_t
        \bigg]
        \leq
        \underbrace{
            \frac{1}{\sqrt{\det{H_t}}} e^{\half\norm{\shift + s_t}_{H_t^{-1}}^2}
        }_{\coloneqq q_t}.
    \end{equation}
    This shows that the process $q_t$ is a non-negative supermartingale; therefore we can apply Lemma~\ref{lemma:supermartingale}.
    Taking the logarithm of both sides of~\eqref{eq:supermartingale-inequality} and using the definition of $q_t$ and $q_0$, we obtain that with probability at least $1 - \delta$, the following holds for all $t \in \I$:
    \begin{align}
        \norm{\shift + s_t}_{H_t^{-1}}^2 \le \norm{\shift + s_0}_{H_0^{-1}}^2
            + 2\log\of{ \frac{\det{H_t}^\half}{\det{H_0}^\half \delta}}.
    \end{align}
    This implies the desired inequality~\eqref{eq:main-result} with uniform-in-time probability at least $1 - \delta$ by substituting $H_t=\bar{P}+V_t$, $H_0=\bar{P}$ and $s_0=0$ and using the following rearrangement:
    \begin{align}
        2\log\of{\!\frac{\det{\bar{P}+V_t}^\half}{\det{\bar{P}}^\half \delta}\!} = \logdet{\Id_\ntheta +\bar{P}^{-1} V_t} + 2\log(1/\delta).
    \end{align}
\end{proof}

\subsection{Consequences for confidence sets}\label{subsection:consequences}

Here we apply Theorem~\ref{thm:main-result} to derive confidence sets in the form of \eqref{eq:confidence-set} for the regression task described in Section~\ref{subsection-model}.
\begin{thm}[Confidence set for $\Pprior \GEQ 0$]\label{thm:general-prior}
    Let Assumptions~\ref{assum:subgaussian} and~\ref{assum:bounded-prior} hold and let $\bar{V} \GEQ 0$ be a deterministic positive semi-definite matrix.
    Then, the finite-sample bound~\eqref{eq:confidence-set} holds with:
    \begin{align}\label{eq:beta-general-prior}
        \beta_t =& \sqrt{ 1 + \rho\of{\of{\Pprior + V_t}^{-1}\bar{V}} }
        \\ & \nonumber \,
        \times \sqrt{ \ctheta^2 + \cw^2 \logdet{\Id_\ntheta +\of{\Pprior+\bar{V}}^{-1} V_t} + 2\cw^2 \log\of{1/\delta}},
    \end{align}
    with the convention that $\beta_t = +\infty$ if either $\Pprior + V_t$ or $\Pprior + \bar{V}$ is not invertible.
\end{thm}
\begin{proof}
    We make use of the following general inequality for matrices $A, B \GEQS 0$ and vector $x \in \R^{\ntheta}$:
    \begin{equation}\label{thm:general-prior:proof:1}
    \norm{x}_{A^{-1}}^2 \le \rho(A^{-1}B) \norm{x}_{B^{-1}}^2.
    \end{equation}
    For $t \in \I$, apply this with $A = \Pprior + V_t$, $B = A+\bar{V}$ and $x = \Pprior (\muprior - \theta^\star) + s_t$ whenever $\Pprior + V_t$ is invertible:
    \begin{align}\label{thm:general-prior:proof:2}
        &\norm{\Pprior (\muprior - \theta^\star) +s_t}_{\of{\Pprior + V_t}^{-1}}^2
        \\ \nonumber & \hspace{4mm}
        \le \rho\of{ A^{-1}\of{A + \bar{V}}}\norm{\Pprior (\muprior - \theta^\star) +s_t}_{\of{A + \bar{V}}^{-1}}^2
        \\ \nonumber & \hspace{4mm}
        = \rho\!\of{ \Id_{\ntheta} + A^{-1}\bar{V}} \norm{\Pprior (\muprior - \theta^\star) +s_t}_{\of{A + \bar{V}}^{-1}}^2
        \\ \nonumber & \hspace{4mm}
        = \underbrace{\rho\!\of{\Id_{\ntheta} + \of{\Pprior + V_t}^{-1}\bar{V}}}
        _{= 1 + \rho\of{\of{\Pprior + V_t}^{-1}\bar{V}}}
        \norm{\Pprior (\muprior - \theta^\star) +s_t}_{\of{\Pprior + \bar{V} + V_t}^{-1}}^2.
    \end{align}
    On the other hand, applying Theorem \ref{thm:main-result} to $\bar{P} = \Pprior + \bar{V}$ and $\shift = \Pprior (\muprior - \theta^\star)$, we get that, with probability at least $1 - \delta$, the following holds true for all $t \in \I$ such that $\Pprior + \bar{V}$ is invertible:
    \begin{align}\label{thm:general-prior:proof:3}
        &\norm{\Pprior (\muprior - \theta^\star) +s_t}_{\of{\Pprior + \bar{V} + V_t}^{-1}}^2
        \\ \nonumber &
        \le \norm{\Pprior (\muprior - \theta^\star)}_{\bar{P}^{-1}}^2 + \cw^2 \logdet{\Id_\ntheta +\bar{P}^{-1} V_t} + 2\cw^2 \log\of{1/\delta}
        \\ \nonumber &
        \le \norm{\muprior - \theta^\star}_{\Pprior}^2 \!+ \! \cw^2 \logdet{\Id_\ntheta +\of{\Pprior + \bar{V}}^{-1} \! V_t} + 2\cw^2 \log\of{1/\delta},
    \end{align}
    \revision{where the second inequality comes from $\Pprior \of{\Pprior + \bar{V}}^{-1} \Pprior \LEQ \Pprior$.}
    Combining \eqref{thm:general-prior:proof:2} and \eqref{thm:general-prior:proof:3} shows that \eqref{eq:confidence-set-reformulated} holds with $\beta_t$ as in \eqref{eq:beta-general-prior} for all $t \in \I$ such that $\Pprior + V_t$ and $\Pprior + \bar{V}$ are invertible.
    If either invertibility condition fails, the result holds trivially as $\beta_t = +\infty$.
    Finally, \revision{we use Assumption~\ref{assum:bounded-prior} to bound $\smallnorm{\muprior - \theta^\star}_{\Pprior}$. Since \eqref{eq:confidence-set-reformulated} is equivalent to \eqref{eq:confidence-set} when $\Pprior + V_t$ is invertible, this concludes the proof.}
\end{proof}
\revision{
    While this general result might be difficult to interpret because of the arbitrary matrix $\bar{V}$, the next three corollaries provide more interpretable results for the special cases $\Pprior \GEQS 0$ and $\Pprior = 0$.
}

\begin{cor}[Confidence set for $\Pprior \GEQS 0$]\label{cor:with-prior}
    Let Assumptions~\ref{assum:subgaussian} and~\ref{assum:bounded-prior} hold and $\Pprior \GEQS 0$.
    Then, the finite-sample bound~\eqref{eq:confidence-set} holds with:
    \begin{equation}\label{eq:main-beta}
        \beta_t = \sqrt{\ctheta^2 + \cw^2\logdet{\Id_\ntheta + \Pprior^{-1}V_t} + 2\cw^2\log\of{1/\delta}}.
    \end{equation}
\end{cor}
\begin{proof}
    Direct application of Theorem~\ref{thm:general-prior} to $\bar{V}=0$.
\end{proof}
Corollary~\ref{cor:with-prior} can only be applied when $\Pprior \GEQS 0$, which requires a prior bound in every parameter direction (see Assumption~\ref{assum:bounded-prior}).
In contrast, the following result applies to OLS (i.e., $\Pprior = 0$) which allows $\ctheta = 0$.

\begin{cor}[Confidence set for $\Pprior = 0$]\label{cor:noprior}
    Let Assumption~\ref{assum:subgaussian} hold, consider the OLS estimator (i.e., $\Pprior = 0$) and let $\bar{V} \GEQS 0$ be an arbitrary deterministic positive definite matrix.
    Then, the finite-sample bound~\eqref{eq:confidence-set} holds with:
    \begin{align}\label{eq:beta-no-prior}
        \beta_t =& \cw \sqrt{ 1 + \rho\of{V_t^{-1}\bar{V}} }\sqrt{\logdet{\Id_\ntheta +\bar{V}^{-1} V_t} + 2\log\of{1/\delta}},
    \end{align}
    with the convention that $\beta_t = +\infty$ when $V_t$ is not invertible.
\end{cor}
\begin{proof}
    Direct application of Theorem~\ref{thm:general-prior} to $\Pprior = 0$.
\end{proof}
\revision{
    In this corollary, it seems desirable to choose $\bar{V} \approx V_t$ for the relevant time steps $t$ to keep $\beta_t$ small.
    This intuition is formalized in the pointwise bound in Corollary~\ref{cor:pointwise} below.
}

\subsection{Pointwise bounds}\label{subsection:pointwise}
The results so far provide uniform-in-time confidence sets.
However, in offline learning, we are only interested in the final estimate $\hat{\theta}_\tmax$, so the uniform-in-time requirement is not necessary.
In this case, we can apply our results with the appropriate choice $\bar{V} \approx \Eof{V_\tmax}$ to derive a tight pointwise bound.
We state this result for OLS only in Corollary~\ref{cor:pointwise} below.
\begin{cor}[A pointwise bound]\label{cor:pointwise}
    Let Assumption~\ref{assum:subgaussian} hold and consider the OLS estimator (i.e., $\Pprior = 0$).
    Assume that $\Eof{V_\tmax} \GEQS 0$ and that for some fixed $\varepsilon_{\tmax} \in [0,1)$ and $\delta' \in [0,1)$:
    \begin{equation}\label{eq:PE}
        \P \bigg[
            (1-\varepsilon_{\tmax}) \Eof{V_\tmax} \LEQ V_\tmax \LEQ (1+\varepsilon_{\tmax}) \Eof{V_\tmax}
        \bigg] \ge 1 - \delta'.
    \end{equation}
    Then, the following holds with probability at least $1 - (\delta + \delta')$:
    \begin{equation}\label{eq:pointwise}
        \norm{\theta^\star - \hat{\theta}_\tmax}_{V_\tmax} \le \cw\sqrt{2 \ntheta \log\of{\frac{2}{1-\varepsilon_{\tmax}}} + 4 \log\of{1/\delta}}.
    \end{equation}
\end{cor}
\begin{proof}
    Let us set $\bar{V} = (1-\varepsilon_{\tmax}) \Eof{V_\tmax} \GEQS 0$ and apply Corollary~\ref{cor:noprior} at $t=\tmax$, resulting in a high probability bound with right-hand side $\beta_\tmax$ as in~\eqref{eq:beta-no-prior}.
    Combining the resulting bound with~\eqref{eq:PE} using a union bound, the two events hold simultaneously with probability at least $1 - (\delta + \delta')$.
    Thus, the eigenvalues of $\bar{V}^{-1}V_{\tmax}$ are in $[1, \frac{1+\varepsilon_{\tmax}}{1-\varepsilon_{\tmax}}]$.\footnotemark
    \footnotetext{This can be seen as $\bar{V}^{-1}V_{\tmax}$ is similar to the positive definite matrix $\bar{V}^{-\half}V_{\tmax}\bar{V}^{-\half}$, which satisfies $\Id_\ntheta \LEQ \bar{V}^{-\half}V_{\tmax}\bar{V}^{-\half} \LEQ \frac{1+\varepsilon_{\tmax}}{1-\varepsilon_{\tmax}} \Id_\ntheta$ by \eqref{eq:PE}.}
    Therefore, we can conclude the proof by bounding $\beta_\tmax$ as follows:
    \begin{align}
            \beta_\tmax
            &\leq \cw\sqrt{ 1 + \rho\of{\Id_\ntheta}}\sqrt{\logdet{\Id_\ntheta +\frac{1+\varepsilon_{\tmax}}{1-\varepsilon_{\tmax}} \Id_\ntheta} + 2\log\of{1/\delta}} \nonumber\\
            &= \cw \sqrt{2\ntheta \log\of{\frac{2}{1-\varepsilon_{\tmax}}} + 4\log\of{1/\delta}}.
    \end{align}
\end{proof}
In system identification, establishing condition~\eqref{eq:PE} is often related to experimental design; see~\cite{tsiamis2023statistical} for a tutorial on this topic.
Interestingly, in the special case of deterministic regressors $M_t$, $V_T$ is also deterministic, so one can always choose $\varepsilon_{\tmax}=0$ and $\delta'=0$.
For this simple case, $\theta^\star -\hat{\theta}_\tmax$ is an arbitrary zero-mean subgaussian vector with proxy $\cw^2 V_\tmax^{-1}$, so the pointwise bound from Corollary~\ref{cor:pointwise} reduces to the following general subgaussian bound.
\begin{cor}[A general subgaussian bound]\label{corollary:subgaussian-concentration}
    For any zero-mean subgaussian variable $v \in \R^{\ntheta}$ with variance proxy $R \GEQS 0$, the following holds with probability at least $1 - \delta$:
    \begin{equation}\label{eq:subgaussian-bound}
        \norm{v}_{R^{-1}}^2 \le 2\ntheta \log(2) + 4\log\of{1/\delta}.
    \end{equation}
\end{cor}
\begin{proof}
    Applying~\eqref{eq:general-subgaussian-bound} with $x=0$ and $H = R$ yields $\smallnorm{v}_{\of{2R}^{-1}}^2 \le \log\det{2\Id_\ntheta} + 2\log\of{1/\delta}$.
    Multiplying both sides by $2$ yields the desired result.
\end{proof}

\section{Discussion}\label{section-discussion}

\subsection{Asymptotic behavior of the bounds}
The bound~\eqref{eq:confidence-set} implies that the estimation error $\smallnorm{\theta^\star - \hat{\theta}_t}$ is bounded by $\sqrt{ \rho\of{(\Pprior + V_t)^{-1}}} \beta_t$ with high probability.
When all eigenvalues of $V_t$ grow linearly with $t$, such a bound decreases as $\O(\beta_t / \sqrt{t})$.
Therefore, in such cases, Theorem~\ref{thm:general-prior} induces error bounds that scale as $\O(\sqrt{ \log t / t})$.
By contrast, the pointwise bound in~\eqref{eq:pointwise} scales as $\O(1/\sqrt{t})$.
The logarithmic factor can therefore be viewed as the ``price to pay'' for the uniform-in-time guarantee.

\revision{
    Next, let us compare our bound for OLS estimates with a classical asymptotic result:
    under some reasonable additional assumptions, $V_\tmax^{1/2} (\hat{\theta}_\tmax - \theta^\star)$ converges in distribution to $\Gausstiny{0}{\Id_\ntheta}$ as ${\tmax \to +\infty}$, see~\cite{Ljung1978asymptotic}.
    Therefore, the smallest valid right-hand side $\beta_{\tmax}$ for a pointwise bound tends to $\sqrt{F_{\chi^2_\ntheta}^{-1}(1-\delta)}$, where $F_{\chi^2_\ntheta}(\cdot)$ is the cumulative density function of the chi-squared distribution with $\ntheta$ degrees of freedom.
    This scales as $\sqrt{\ntheta}$ for $\ntheta \to +\infty$ and as $\sqrt{2\log(1/\delta)}$ for $\delta \to 0$.
    This aligns well with the present finite-sample pointwise result in Corollary~\ref{cor:pointwise}, where $\beta_{\tmax} = \O(\sqrt{\ntheta + 2\log(1/\delta)})$.
}

Regarding uniform-in-time bounds, an alternative approach is to combine pointwise bounds with failure probabilities $\delta_t$ using a union bound to obtain a uniform-in-time guarantee with failure probability $\delta = \sum_{t=1}^{\tmax} \delta_t$.
For a pointwise bound that scales as $\O(\sqrt{\log(1/\delta)})$ as in Corollary~\ref{cor:pointwise}, applying this principle with $\delta_t \propto 1/t^2$ also obtains a uniform-in-time bound that scales as $\O(\sqrt{ \log t / t})$.
This is further evidence that the logarithmic growth in $t$ is connected to the uniform-in-time guarantee.

\subsection{Relation to existing results}
\revision{
    The seminal result~\cite{abbasi2011improved} introduced SNM-based finite-sample bounds for scalar outputs and \cite{chowdhury2021noregret} generalized this result to vector-valued outputs.
    Many subsequent works have built upon the scalar result~\cite{abbasi2011regret,tsiamis2023statistical,ziemann2023tutorial,matni2019tutorial,wagenmaker2020active,lew2022safe,koller2018learning,sui2015safe,chowdhury2017kernelized}, but the vector-valued generalization seems to have drawn less attention.
    The present work is conceptually similar to the finite-sample bound in~\cite{chowdhury2021noregret} but contains a more extensive analysis of different variations, similar to how~\cite{tsiamis2023statistical,ziemann2023tutorial,matni2019tutorial} developed system identification results based on~\cite{abbasi2011improved}.
}
To avoid any confusion, we clarify that the term ``vector-valued self-normalized martingale'' often used in the literature refers to vector-valued parameters, but scalar noise and hence scalar outputs.

Importantly, existing results~\cite{abbasi2011improved,chowdhury2021noregret} treat the regularization bias separately with a triangle inequality which leads to the following right-hand side in the finite-sample bound:
\begin{equation}\label{eq:abbasi2011improved}
    \tilde{\beta}_t = \ctheta + \cw\sqrt{\logdet{\Id_\ntheta + \Pprior^{-1}V_t} + 2\log\of{1/\delta}}.
\end{equation}
Compared to the bound $\beta_t$ in the present Corollary~\ref{cor:with-prior}, $\tilde{\beta}_t$ is always greater or equal to $\beta_t$, but in the worst case by a factor \revision{$\sqrt{2}$}, and becomes similar if one of the terms dominates the others.
Thus, even in the scalar case, the new bound is tighter than the state-of-the-art bound.

The results of this paper are only proven for finite-dimensional parameter spaces, but a generalization of these results to Hilbert spaces, as in~\cite{chowdhury2021noregret}, seems possible.
Such a generalization is an interesting direction for future work, as it would make the present results applicable to kernel-based regression as well, which received quite some attention in the recent literature~\cite{chowdhury2017kernelized,koller2018learning,sui2015safe,lahr2026optimal}.

\subsection{Application to system identification}\label{subsection:sysid}
The main use case of the present results is to derive finite-sample bounds for system identification problems with full-state measurements and affine-in-parameter dynamics:
\begin{equation}\label{eq:dynamical-system}
    x_{t+1} = \overbrace{\bar{f}(x_t, u_t) + M(x_t, u_t)\theta^\star}^{\eqqcolon f(x_t, u_t; \theta^\star)} + w_t, \quad \forall t \in \Iminus,
\end{equation}
where $x_t \in \R^{\nx}$ is the state, $u_t \in \R^{\nuu}$ is the input, $w_t \in \R^{\nx}$ is the disturbance,
\revision{$\bar{f}(\cdot, \cdot)$ is a known nominal model, $M(\cdot, \cdot)$ is a known regressor matrix, and $\theta^\star \in \R^{\ntheta}$ is the unknown parameter to be estimated.}
The identification task can be formulated as in~\eqref{eq:model} with $M_t \coloneqq M(x_t, u_t)$ and $y_{t+1} \coloneqq x_{t+1} - \bar{f}(x_t, u_t)$.

For this application, the output bound~\eqref{eq:output-bound} becomes a bound on the one-step prediction error:
\begin{align}\label{eq:prediction-error-bound}
    \P \bigg[
        \norm{f(x, u; \theta^\star) - f(x, u; \hat{\theta}_t)} \le \beta_t \sigma_t(x, u),&
        \\ \nonumber & \hspace{-35mm}
        \forall t \in \I, \forall (x,u) \in \R^{\nx + \nuu}
    \bigg] \ge 1 - \delta,
\end{align}
with $\sigma_t(x, u) \coloneqq \smallonorm{M(x,u)\of{\Pprior + V_t}^{-1/2}}$.
Importantly, even if the noise $w_t$ is independently generated, it can influence the future regressors $M_{t+k}$ for $k \geq 1$ through the state evolution or the feedback mechanisms, which is covered by our assumptions, as discussed in Section~\ref{subsection:assumptions}.
The SNM finite-sample bounds in system identification presented in~\cite{tsiamis2023statistical,ziemann2023tutorial,matni2019tutorial} are restricted to matrix regression structure, so they are not applicable when some parameters enter in several components of the dynamics.
The present bounds overcome this limitation: they apply to the general setup~\eqref{eq:dynamical-system}.

\subsection{Structure-agnostic identification of linear systems}\label{subsection:lti}
A system identification task of interest in the literature is the structure-agnostic identification of a Linear Time-Invariant (LTI) system from full-state measurements~\cite{tsiamis2023statistical,ziemann2023tutorial,wagenmaker2020active,abbasi2011regret,matni2019tutorial}:
\begin{align}\label{eq:lti-system}
    \underbrace{x_{t+1}}_{= y_{t+1}}
    = A^\star x_t + B^\star u_t + w_t
    = \underbrace{[A^\star \quad B^\star] }_{\eqqcolon \Theta^\star} \underbrace{\begin{bmatrix} x_t \\ u_t \end{bmatrix}}_{\eqqcolon \varphi_t} + w_t,
\end{align}
We focus on OLS as it is more common in the related literature:
the matrix-valued parameter $\Theta^\star \in \R^{\nx \times \nphi}$ ($\nphi \coloneqq \nx + \nuu$) is estimated as follows:
\begin{align}\label{eq:OLS-lti}
    \hat{\Theta}_t &\coloneqq
    \underset{\Theta}{\arg\min}
    \, \sum_{k=0}^{t-1} \norm{y_{k+1} - \Theta \varphi_k}_W^2
    = \sum_{k=0}^{t-1} y_{k+1} \varphi_k^\top \Phi_t^{-1},
\end{align}
with $\Phi_t \coloneqq \sum_{k=0}^{t-1} \varphi_k \varphi_k^\top$.
\revision{
    Note that the closed-form formula in~\eqref{eq:OLS-lti} is well-defined only when $\Phi_t$ is invertible.
    Surprisingly, it is independent of $W \GEQS 0$.}
This use case is covered by the general setup~\eqref{eq:model} by setting $\theta = \vec{\Theta}$ and ${M_t = \Id_\nx \otimes \varphi_t^\top}$, therefore we can apply Corollary~\ref{cor:noprior} to this special case.
\begin{cor}[Bounds for LTI identification]\label{corollary:lti}
    Consider the LTI model~\eqref{eq:lti-system} with \revision{$w_t$ conditionally subgaussian as in Assumption~\ref{assum:subgaussian} with $W=\Id_\ny$.
    Let $ \hat{\Theta}_t$ be the OLS estimates~\eqref{eq:OLS-lti}} and let $\bar{\Phi} \GEQS 0$ be a positive definite matrix.
    Then, the following holds:
    \begin{align}\label{eq:lti-confidence-set}
        \P \bigg[
            \fnorm{ \of{\Theta^\star - \hat{\Theta}_t} \Phi_t^\half} \!\le \beta_t^{\mr{LTI}}\!, \; \forall t \in \I \bigg] \ge 1 - \delta,\!\!
    \end{align}
    where $\beta_t^{\mr{LTI}}$ is given by:
    \begin{align}\label{eq:lti-beta}
        \beta_t^{\mr{LTI}}&= \cw \sqrt{1+\rho\of{\Phi_t^{-1} \bar{\Phi}}}
            \\ \nonumber&\hspace{15mm}
            \sqrt{\nx \logdet{\Id_{\nphi} + \bar{\Phi}^{-1}\Phi_t} + 2\log\of{1/\delta}},
    \end{align}
    with $\rho\of{\Phi_t^{-1} \bar{\Phi}}=+\infty$ whenever $\Phi_t$ is not invertible.
\end{cor}
\begin{proof}
    Apply Corollary~\ref{cor:noprior} to the vectorized version of the problem, which satisfies $V_t = \Id_\nx \otimes \Phi_t$.
    The resulting bound can be formulated as~\eqref{eq:lti-confidence-set}-\eqref{eq:lti-beta} using the following properties of the Kronecker product:
    \begin{subequations}
    \begin{align}
        &\norm{\vec{\Theta^\star - \hat{\Theta}_t}}_{\Id_\nx \otimes \Phi_t}^2 = \fnorm{\of{\Theta^\star - \hat{\Theta}_t} \Phi_t^\half}^2, \\
        & \mbox{for any $\tilde{\Phi} \in \R^{\nphi \times \nphi}$:} \; \bigdet{\Id_\nx \otimes \tilde{\Phi}} = \bigdet{\tilde{\Phi}}^{\nx}.
    \end{align}
    \end{subequations}
\end{proof}
This result was proven for pointwise probability bounds only in~\cite[Theorem IV.1 (first part)]{ziemann2023tutorial}.
In~\cite[Theorem 1]{abbasi2011regret} or~\cite[equation (S3)]{tsiamis2023statistical}, a looser bound is derived in similar settings by applying univariate SNM bounds on each state component and combining them with a union bound.
Using the inequalities $\smallnorm{E \tilde{\varphi}} \le \onorm{E} \smallnorm{\tilde{\varphi}} \le \fnorm{E} \smallnorm{\tilde{\varphi}}$ on an appropriate matrix $E$ and vector $\tilde{\varphi}$, one can use~\eqref{eq:lti-confidence-set} to bound the one-step-ahead predictions, as in~\eqref{eq:prediction-error-bound}:
\begin{align}\label{eq:prediction-error-bound-lti}
    \P \bigg[
        \norm{\Theta^\star \varphi - \hat{\Theta}_t \varphi } \le & \beta_t^{\mr{LTI}} \norm{\varphi}_{\Phi_t^{-1}},
        \\ \nonumber & \hspace{10mm}
        \forall t \in \I, \forall \varphi \in \R^{\nphi}
    \bigg] \ge 1 - \delta.
\end{align}
Another common approach (e.g. in~\cite[Theorem IV.1 (second part)]{ziemann2023tutorial}) is to bound the operator norm instead of the Frobenius norm in~\eqref{eq:lti-confidence-set} with the following right-hand side:
\begin{align}\label{eq:lti-bound-operator-norm}
    \tilde{\beta}_t^{\mr{LTI}} = 2\cw \sqrt{1+\rho\of{\Phi_t^{-1} \bar{\Phi}}}&
    \\ \nonumber&\hspace{-27mm}
    \sqrt{\logdet{\Id_{\nphi} + \bar{\Phi}^{-1}\Phi_t} + 2\nx \log\of{5} + 2\log\of{1/\delta}}.
\end{align}
The proof of this bound relies on combining univariate SNM bounds with a covering argument.
When only output bounds of the form~\eqref{eq:prediction-error-bound-lti} are desired, this might be preferable for large-scale systems.
Indeed, while the operator norm is always less or equal to the Frobenius norm, one can directly replace $\beta_t^{\mr{LTI}}$ by $\tilde{\beta}_t^{\mr{LTI}}$ in \eqref{eq:prediction-error-bound-lti}, which becomes smaller for large values of $\nx$ and $\det{\Id_{\nphi} + \bar{\Phi}^{-1}\Phi_t}$ (see~\cite[Remark IV.1]{ziemann2023tutorial}).
The operator-norm bound in~\cite{ziemann2023tutorial} was proven for pointwise probability bounds only, but there does not seem to be any fundamental difficulty in adapting the proof for uniform-in-time bounds.

\section{Numerical experiments}\label{section-numerical-experiments}
This section illustrates the proposed finite-sample bounds on two examples\footnotemark.
The first one is a scalar polynomial regression problem, used to compare the new results with the classical scalar-output bound from~\cite{abbasi2011improved}.
The second one is a multi-output system-identification problem, used to demonstrate the practical benefit of the new results: they are applicable to a structure-exploiting identification approach, whereas the existing results can only be applied to a less efficient structure-agnostic identification approach.
\footnotetext{
    All experiments are reproducible using the Python code associated with this paper, available at
    \underline{\url{https://github.com/Leo-Simpson/finite-sample-bounds}}
}

\subsection{Example 1: A scalar polynomial estimation task}
\label{subsec:exp_poly}

We consider the following polynomial identification task:
\begin{equation}\label{eq:SO-example}
    y_{t+1} = \overbrace{\theta^\star_0 + \theta^\star_1 u_t + \theta^\star_2 u_t^2 + \theta^\star_3 u_t^3}^{\eqqcolon g(u_t;\theta^\star)} \; + w_t,
\end{equation}
with $w_t \sim \Gauss{0}{\cw^2}$ and $\cw=0.2$.
The inputs are generated by the following feedback law
\begin{align}
    u_t = \mr{clip}\of{r_t - \sum_{k=1}^t y_k}, \qquad
    r_t \coloneqq 2 \sin(0.1 t + \phi),
\end{align}
where $\mr{clip}(\cdot)$ is the projection to the interval $[-1,1]$ and $\phi$ is generated uniformly at random in $[0,2\pi]$ on each realization.
This closed-loop data-generation mechanism induces intractable regressor-noise correlations, which motivates the use of SNM bounds.
To generate data, we use $g(u;\theta^\star) \coloneqq u - u^3$ and $\tmax = 20$.
The estimates $\hat{\theta}_t$ are computed from~\eqref{eq:opti-estimates} with $W=1$ and with a regularization defined by $\mu=0$ and $\smallnorm{\theta}^2_{\Pprior} = \int_{-1}^{1} g(u;\theta)^2 \,\mr{d}u$.

\begin{figure}[!t]
    \centerline{\includegraphics[width=\columnwidth]{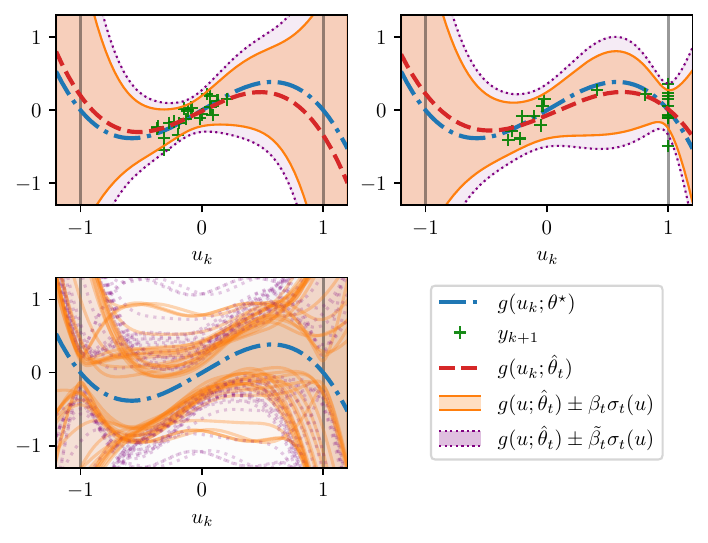}}
	\caption{
        \revision{
            Illustration of example~\eqref{eq:SO-example}.
            The three plots show output bounds~\eqref{eq:output-bound} at $t=T=20$ for $\delta=0.05$ and $\ctheta = \norm{\theta^\star}_{\Pprior}$.
            These are depicted for both the $\beta_t$ from Corollary~\ref{cor:with-prior} (orange) and for $\tilde{\beta}_t$ from~\cite[Theorem 2]{abbasi2011improved} (purple).
            The top plots display examples of realizations, while the bottom plot overlays the output confidence bands from $20$ independent runs.
        }
    }\label{fig-illustrative-runs}
\end{figure}
Figure~\ref{fig-illustrative-runs} shows the output bounds~\eqref{eq:output-bound} obtained using Corollary~\ref{cor:with-prior} with $\delta=0.05$ and $\ctheta = \norm{\theta^\star}_{\Pprior}$.
The true function $g(\cdot;\theta^\star)$ is contained in more than $95\%$ of these confidence bands, consistent with the theory.
\revision{
    This figure also compares the novel bound from Corollary~\ref{cor:with-prior} with right-hand side $\beta_t$ to the more conservative bound from~\cite[Theorem 2]{abbasi2011improved} repeated here in~\eqref{eq:abbasi2011improved} with right-hand side $\tilde{\beta}_t$.
}

\begin{figure}[!t]
    \centerline{\includegraphics[width=\columnwidth]{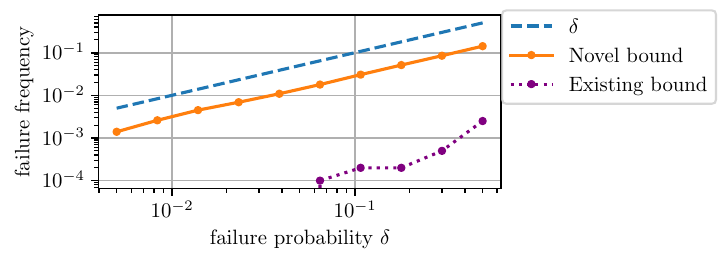}}
	\caption{
        Empirical failure frequency of the finite-sample bounds from Corollary~\ref{cor:with-prior} (``Novel bound'') and ~\cite[Theorem 2]{abbasi2011improved} (``Existing bound'') as a function of $\delta$.
    }\label{fig-violation-vs-delta}
\end{figure}
\revision{
    Next, we compute the empirical failure frequency of the bounds over $10^4$ independent realizations, i.e., the fraction of realizations for which $\smallnorm{\hat{\theta}_t - \theta^\star}_{\Pprior+V_t} > \beta_t$ for some $t\in \I$.
}
Figure~\ref{fig-violation-vs-delta} reports these frequencies as a function of $\delta$, for the bound given by $\beta_t$ (``Novel bound'') and the one given by $\tilde{\beta}_t$ (``Existing bound'').

The observed failure frequencies remain below the \revision{failure probability} $\delta$, as predicted by the theory.
Importantly, although the reduction from $\tilde{\beta}_t$ to $\beta_t$ might seem marginal in Figure~\ref{fig-illustrative-runs},
Figure~\ref{fig-violation-vs-delta} demonstrates that the probability bound of interest is much tighter for the novel radius $\beta_t$.

\subsection{Example 2: A system identification problem}\label{subsection:example-multi}
\revision{
    Next, we study a system identification example that demonstrates the practical benefit of formulating finite-sample bounds for more efficient identification approaches than the structure-agnostic LTI identification approach.
}
We consider a heat-transfer model with $\nx=5$ internal temperature states $x_t \coloneqq \big[ \vartheta^{[1]}_t \; \dots \; \vartheta^{[\nx]}_t \big]^\top$ and two exogenous boundary temperatures $u_t \coloneqq \big[ \vartheta^{[0]}_t \quad \vartheta_t^{\mr{ext}} \big]^\top$, interacting through convection.
The dynamics are modeled as follows:
\begin{align}\label{eq:heat-exchange}
    &\text{for } i = 1, \dots, \nx: \\ \nonumber
    & \quad \vartheta^{[i]}_{t+1} = \vartheta^{[i]}_t + \alpha \of{\vartheta^{[i-1]}_t \!\! - \vartheta^{[i]}_t} + \beta \of{\vartheta_t^{\mr{ext}} \!\!- \vartheta^{[i]}_t} + w^{[i]}_t,
\end{align}
with unknown parameters $\alpha, \beta \in (0, 1)$ and noise terms $w^{[i]}_t$ independent and uniformly distributed in $[-1,1]$.
Thus, for $W=\Id_{\nx}$, the noise vectors $w_t \in \R^{\nx}$ satisfy Assumption~\ref{assum:subgaussian} with $\cw = \frac{1}{\sqrt{3}}$ (this is sharper than the constant $\cw=1$ given by Hoeffding's inequality).
The ambient temperature is kept constant $\vartheta_t^{\mr{ext}}=25 ~^{\circ}\text{C}$ and the temperature of the first node excites the system: \revision{$\vartheta^{[0]}_t = \vartheta^{\mr{in}} \of{ 1 + \sin(0.1 t)}$}, with $\vartheta^{\mr{in}} = 100 ~^{\circ}\text{C}$.
The initial temperatures are set to $\vartheta^{[i]}_0 = 20 ~^{\circ}\text{C}$.
The true (unknown) parameters are $\alpha^\star = 0.5$, $\beta^\star = 0.1$.

This system is an affine-in-parameter model in the form of~\eqref{eq:dynamical-system}, which can be identified by estimating $\theta \coloneqq [ \alpha \quad \beta ]^\top$ directly.
This is done with OLS (i.e., \eqref{eq:opti-estimates} with $W = \Id_\nx$ and $\Pprior = 0$) and we apply Corollary~\ref{cor:noprior} to these estimates with $\bar{V} = \of{\vartheta^{\mr{in}}}^2 \Id_2$, $\delta=0.05$, $\ctheta=0$, and $\cw = \frac{1}{\sqrt{3}}$.
The resulting finite-sample bounds translate into one-step-ahead prediction error bounds as in~\eqref{eq:prediction-error-bound}.
At each time $t$, we consider the worst-case bound over $(x, u) \in \cl{S}$, where $\cl{S}$ is a set consisting of $10^3$ state-inputs pairs generated uniformly at random in $[0, 2\vartheta^{\mr{in}}]^{\nx+2}$.
These bounds read as follows:
\begin{equation}\label{eq:prediction-bound-max}
    \underbrace{
        \max_{(x, u) \in \cl{S}} \!\norm{f(x, u; \theta^\star) - f(x, u; \hat{\theta}_t)}
    }_{\text{l.h.s. in param. id.}}
    \le
    \underbrace{
        \beta_t \!\max_{(x, u) \in \cl{S}} \!\sigma_t(x, u)
    }_{\text{r.h.s. in param. id.}},
\end{equation}
where $f(x, u; \theta)$ models the system~\eqref{eq:heat-exchange} and $\sigma_t(x, u)$ is computed as in~\eqref{eq:prediction-error-bound}.

By contrast, the existing scalar-output bounds do not apply, because each parameter affects several state components simultaneously.
A natural baseline is therefore the structure-agnostic LTI identification approach discussed in Section~\ref{subsection:lti}.
For this approach, we apply the bound from~Corollary~\ref{corollary:lti} with $\bar{\Phi} = \of{\vartheta^{\mr{in}}}^2 \Id_{\nx+2}$ and $\delta=0.05$ to the same data.
As above, we translate these bounds into worst-case one-step-ahead prediction error bounds over the same set $\cl{S}$ using~\eqref{eq:prediction-error-bound-lti}:
\begin{equation}\label{eq:prediction-bound-max-lti}
    \underbrace{
        \max_{\varphi \in \cl{S}} \norm{\of{\Theta^\star - \hat{\Theta}_t}\varphi}
    }_{\text{l.h.s. in LTI id.}}
    \le
    \underbrace{
        \beta_t^{\mr{LTI}} \max_{\varphi \in \cl{S}} \norm{\varphi}_{\Phi_t^{-1}},
    }_{\text{r.h.s. in LTI id.}}
\end{equation}
with $\varphi \coloneqq [x \quad u]^\top$, $\Theta^\star \coloneqq [A^\star \, B^\star]$ where $A^\star, B^\star$ are the matrices in the LTI form of the system~\eqref{eq:heat-exchange} at $\theta=\theta^\star$, and $\hat{\Theta}_t \coloneqq [\hat{A}_t \, \hat{B}_t]$ are the corresponding estimates.
As discussed in Section~\ref{subsection:lti}, we can also replace $\beta_t^{\mr{LTI}}$ by the bound on the operator norm $\tilde{\beta}_t^{\mr{LTI}}$ defined in~\eqref{eq:lti-bound-operator-norm}.
\revision{Note that to keep the comparison as neutral as possible, the matrices $\bar{V}$ and $\bar{\Phi}$ are chosen based only on their physical unit (that is, the square of some temperature).}

\begin{figure}[!t]
    \centerline{\includegraphics[width=\columnwidth]{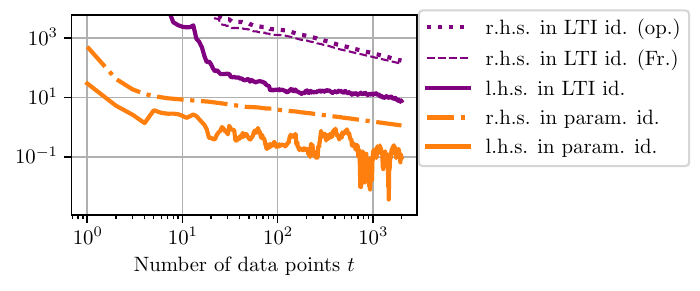}}
	\caption{
        Left and right-hand sides of the bounds~\eqref{eq:prediction-bound-max} and~\eqref{eq:prediction-bound-max-lti} for the heat-transfer example~\eqref{eq:heat-exchange}.
    }\label{fig-heat-exchange}
\end{figure}
Figure~\ref{fig-heat-exchange} compares the one-step-ahead error bounds~\eqref{eq:prediction-bound-max} and~\eqref{eq:prediction-bound-max-lti} as a function of time $t$.
The orange lines correspond to inequality~\eqref{eq:prediction-bound-max} (which results from Corollary~\ref{cor:noprior}).
The purple line ``Fr.'' corresponds to inequality~\eqref{eq:prediction-bound-max-lti} with $\beta_t$ as in Corollary~\ref{corollary:lti}, while the purple line ``op.'' corresponds to the variation that uses $\tilde{\beta}_t^{\mr{LTI}}$ instead of $\beta_t^{\mr{LTI}}$ (i.e., the bound that applies to the operator norm).

\revision{
    This figure shows that the structure-exploiting formulation yields output error bounds about two orders of magnitude smaller than those obtained with the structure-agnostic LTI formulation.
    This improvement is primarily due to the more efficient learning formulation enabled by the proposed result: the corresponding left-hand side is itself substantially smaller.
    This is not surprising since the structure-exploiting method estimates only the two physical parameters while the structure-agnostic LTI identification estimates the $35$ entries of the matrices $A^\star$ and $B^\star$.
}
Moreover, it also shows that the bounds are reasonably tight relative to their respective left-hand sides, particularly for the proposed structure-exploiting formulation.
The two LTI-based bounds obtained from~\eqref{eq:lti-beta} (bound on the Frobenius norm) and~\eqref{eq:lti-bound-operator-norm} (bound on the operator norm) exhibit similar behavior in this example, with the former being slightly tighter.

Overall, the experiment highlights the practical value of the new multi-output result: it enables finite-sample guarantees for structure-exploiting identification, instead of restricting these guarantees to structure-agnostic identification, which is often less efficient.

\section{Conclusion}\label{section-conclusion}
We derived uniform-in-time finite-sample bounds for regularized linear regression with vector-valued outputs and conditionally zero-mean subgaussian noise.
By revisiting how the standard martingale arguments are used, we obtained tighter bounds than the state of the art.
More importantly, the proposed approach applies directly to multi-output regression and therefore \revision{extends the applicability of finite-sample bounds used in system identification.}
As illustrated in the numerical example, this is a substantial advantage, since the bounds can be applied to structure-exploiting identification approaches.

Promising directions for future work include applying these bounds in decision-making settings, for example, in model-based control with online learning of model parameters.

%% file: appendix.tex
\renewcommand{\nu}{u} % here we use u instead of \nu to avoid confusion with the variable v.
\begin{proof}[Proof of Lemma \ref{lemma:subgaussian}]
    Let $v \in \R^{\ntheta}$ be a zero-mean subgaussian random variable with variance proxy $R \GEQ 0$, i.e., it satisfies:
    \begin{equation}\label{eq:proof-subgaussian-1}
        \forall \nu \in \R^{\ntheta}, \quad
        \Eof{\exp{\nu^\top v}} \le \exp{\frac{1}{2} \norm{\nu}^2_R}.
    \end{equation}
    Let $H \GEQS 0$ and $x \in \R^{\ntheta}$ be fixed.
    Multiplying the inequality \eqref{eq:proof-subgaussian-1} by $\exp{-\frac{1}{2} \norm{\nu}_{H+R}^2 + \nu^\top x}$ and integrating both sides over $\nu \in \R^{\ntheta}$ yields:
    \begin{equation}\label{eq:proof-subgaussian-2}
        \Eof{\int_{\R^{\ntheta}}\hspace{-2mm}\exp{-\frac{1}{2}\norm{\nu}_{H+R}^2 + \nu^\top (x+v)  }  \mr{d}\nu } \le \hspace{-1mm}\int_{\R^{\ntheta}}\hspace{-2mm}\exp{-\frac{1}{2} \norm{\nu}_H^2 + \nu^\top x }  \mr{d}\nu
    \end{equation}
    To compute these integrals, we use the following integral formula derived from the normalizing constant in the Gaussian distribution $\Gausstiny{\Sigma b}{\Sigma}$:
    \begin{equation}\label{eq:gauss-integral}
        \forall b, \Sigma: \int_{\R^n} \exp{-\frac{1}{2} \norm{\nu}_{\Sigma^{-1}}^2 + \nu^\top b} \mr{d}\nu =
        \sqrt{(2 \pi)^n \det{\Sigma}} \exp{\frac{1}{2} \norm{b}_{\Sigma}^2 }.
    \end{equation}
    Using this formula on both sides of the inequality \eqref{eq:proof-subgaussian-2} with $\Sigma = (H+R)^{-1}$ for the left-hand side and $\Sigma=H^{-1}$ for the right-hand side, we obtain:
    \begin{align}\label{eq:proof-subgaussian-3}
        \Eof{\sqrt{\frac{(2 \pi)^n}{\det{H+R}}} \exp{\frac{1}{2} \norm{x+v}_{(H+R)^{-1}}^2 } \hspace{-1.1mm}} &\le
        \sqrt{\frac{(2 \pi)^n}{\det{H}}} \exp{\frac{1}{2} \norm{x}_{H^{-1}}^2 }.
    \end{align}
    Eliminating the common factor $\sqrt{(2 \pi)^{\ntheta}}$ on both sides, we obtain the desired property \eqref{eq:subgaussian-concentration}.
    %  \begin{equation}
    %     \frac{1}{\sqrt{\det{H+R}}} \Eof{e^{\half\norm{x + v}_{\of{H + R}^{-1} }^2} }
	%     \leq \frac{1}{\sqrt{\det{H}}} e^{\half\norm{x}_{H^{-1}}^2}
    % \end{equation}
\end{proof}

\ifextendedversion {
\begin{proof}[Proof of Lemma \ref{lemma:supermartingale}]
    Let $q_0, \dots, q_{\tmax}$ be non-negative supermartingale, i.e., a sequence of non-negative random variables such that, $q_t$ is a function of $\cl{D}_t$ only, and satisfies $q_t \ge \Eof{q_{t+1} \mid \cl{D}_t}$.
    Let $\delta \in (0, 1)$ be a scalar.
    Define the event $\event_t \coloneqq \curlyof{ q_k \le \frac{\Eof{q_0}}{\delta}, \; \forall k \leq t}$, which is a function of $\cl{D}_t$ only.
    Then, for $t \in \I$, define the auxiliary random variable $a_t$ according to:
    \begin{equation}\label{eq:proof-supermartingale-1}
        a_0 = q_0, \quad a_{t+1} = \begin{cases} q_{t+1} & \text{if } \event_t, \\ a_t & \text{otherwise}. \end{cases}
    \end{equation}
    On the one hand, this implies:
    \begin{equation}\label{eq:proof-supermartingale-2}
        \begin{aligned}
            \Eof{a_{t+1} \mid \event_t } &= \Eof{q_{t+1} \mid \event_t } \\
            &= \Eof{ \Eof{q_{t+1} \mid \event_t, \cl{D}_t} \mid \event_t } \\
            &= \Eof{ \Eof{q_{t+1} \mid \cl{D}_t} \mid \event_t } \\
            &\le \Eof{ q_t \mid \event_t } = \Eof{ a_t \mid \event_t }.
        \end{aligned}
    \end{equation}
    On the other hand, we also have $\Eof{a_{t+1} \mid \event_t^{\text{c}} } = \Eof{a_t \mid \event_t^{\text{c}} }$, where $\event_t^{\text{c}}$ denotes the complement of the event $\event_t$.
    These two facts imply that $\Eof{a_{t+1}} \le \Eof{a_t}$.
    Repeating this argument for all $t = 0, \dots, \tmax - 1$, we obtain $\Eof{a_{\tmax}} \le \Eof{a_0} = \Eof{q_0}$.
    Thus, applying Markov's inequality to $a_{\tmax}$ yields:
    \begin{equation}\label{eq:proof-supermartingale-3}
        \Pof{ a_{\tmax} \le \frac{\Eof{q_0}}{\delta} } \ge 1 - \delta.
    \end{equation}
    To conclude, note that
    \begin{equation}\label{eq:proof-supermartingale-4}
        \of{a_{\tmax} \le \frac{\Eof{q_0}}{\delta}} \iff \of{q_t \le \frac{\Eof{q_0}}{\delta} \, \forall t \in \I}.
    \end{equation}
    Combining this with \eqref{eq:proof-supermartingale-3} yields the desired result
    \begin{equation}
        \Pof{ q_t \le \frac{\Eof{q_0}}{\delta} \quad \forall t \in \I } \ge 1 - \delta.
    \end{equation}
\end{proof}
}
\fi

%% file: biblio.bib
@String { CDC      = {Conf. on Decision and Control (CDC)} }

@String { TOR     = {IEEE Trans. on Robotics}}

@String { magazin  = {IEEE Control Systems Magazine} }

@inproceedings{abbasi2011improved,
 author = {Abbasi-Yadkori, Yasin and P\'{a}l, D\'{a}vid and Szepesv\'{a}ri, Csaba},
 booktitle = {Advances in Neural Information Processing Systems},
 pages = {},
 title = {Improved Algorithms for Linear Stochastic Bandits},
 volume = {24},
 year = {2011}
}

@InProceedings{abbasi2011regret,
  title = 	 {Regret Bounds for the Adaptive Control of Linear Quadratic Systems},
  author =       {Abbasi-Yadkori, Yasin and Szepesv\'ari, Csaba},
  booktitle = 	 {Proceedings of the 24th Annual Conf. on Learning Theory},
  pages = 	 {1--26},
  year = 	 {2011},
  volume = 	 {19},
  publisher = {PMLR},
}

@article{baumgartner2026gray,
  title={Gray-Box Optimization using Optimism in the Face of Uncertainty},
  author={Baumg{\"a}rtner, Katrin and Simpson, L{\'e}o and Diehl, Moritz},
  journal={arXiv preprint arXiv:2606.17726},
  year={2026}
}

@Inbook{boucheron2004concentration,
  title={Concentration inequalities},
  author={Boucheron, St{\'e}phane and Lugosi, G{\'a}bor and Bousquet, Olivier},
  booktitle={Advanced Lectures on Machine Learning},
  pages={208--240},
  year={2004},
  publisher={Springer},
  address={Berlin, Heidelberg},
}

@article{care2017finite,
  title={Finite-sample system identification: An overview and a new correlation method},
  author={Care, Algo and Cs{\'a}ji, Bal{\'a}zs Cs and Campi, Marco C and Weyer, Erik},
  journal={IEEE Control Systems Letters},
  volume={2},
  number={1},
  pages={61--66},
  year={2017},
  publisher={IEEE}
}

@InProceedings{chowdhury2017kernelized,
  title = 	 {On Kernelized Multi-armed Bandits},
  author =       {Sayak Ray Chowdhury and Aditya Gopalan},
  booktitle = 	 {Proceedings of the 34th International Conf. on Machine Learning},
  pages = 	 {844--853},
  year = 	 {2017},
  volume = 	 {70},
  publisher =    {PMLR},
}

@book{DeLaPena2009self,
  title={Self-normalized processes: Limit theory and statistical applications},
  author={De la Pena, Victor H and Lai, Tze Leung and Shao, Qi-Man},
  year={2009},
  serie={Probability and Its Applications},
  publisher={Springer}
}

@article{hou2013model,
  title={From model-based control to data-driven control: Survey, classification and perspective},
  author={Hou, Zhong-Sheng and Wang, Zhuo},
  journal={Information Sciences},
  volume={235},
  pages={3--35},
  year={2013},
  publisher={Elsevier}
}

@inproceedings{koller2018learning,
  title={Learning-based model predictive control for safe exploration},
  author={Koller, Torsten and Berkenkamp, Felix and Turchetta, Matteo and Krause, Andreas},
  booktitle=CDC,
  pages={6059--6066},
  year={2018},
  organization={IEEE}
}

@article{kosut1992set,
  title={Set-membership identification of systems with parametric and nonparametric uncertainty},
  author={Kosut, Robert L and Lau, Ming K and Boyd, Stephen P},
  journal={IEEE Transactions on Automatic Control},
  volume={37},
  number={7},
  pages={929--941},
  year={1992},
  publisher={IEEE}
}

@article{lahr2026optimal,
  title={Optimal uncertainty bounds for multivariate kernel regression under bounded noise: A {G}aussian process-based dual function},
  author={Lahr, Amon and Scampicchio, Anna and K{\"o}hler, Johannes and Zeilinger, Melanie N},
  journal={IEEE Control Systems Letters},
  year={2026},
  publisher={IEEE},
  volume={10},
  pages={1963-1968},
}

@article{lew2022safe,
  title={Safe active dynamics learning and control: A sequential exploration--exploitation framework},
  author={Lew, Thomas and Sharma, Apoorva and Harrison, James and Bylard, Andrew and Pavone, Marco},
  journal=TOR,
  volume={38},
  number={5},
  pages={2888--2907},
  year={2022},
  publisher={IEEE}
}

@Book{ljung1999system,
  Title                    = {{S}ystem identification: Theory for the User},
  Author                   = {L. Ljung},
  Publisher                = {Prentice Hall},
  Year                     = {1999},
  Address                  = {Upper Saddle River, N.J.},
}

@InProceedings{Ljung1978asymptotic,
  Title                    = {Asymptotic normality of prediction error estimators for approximate system models},
  Author                   = {Ljung, Lennart and Caines, Peter E},
  Booktitle                = CDC,
  Pages                    = {927--932},
  Year                     = {1978},
  Organization             = {IEEE}
}

@inproceedings{matni2019tutorial,
  title={A tutorial on concentration bounds for system identification},
  author={Matni, Nikolai and Tu, Stephen},
  booktitle=CDC,
  pages={3741--3749},
  year={2019},
  organization={IEEE}
}

@book{shiryaev2019probability,
  title={Probability-2},
  author={Shiryaev, Albert N},
  year={2019},
  serie={Graduate Texts in Mathematics},
  publisher={Springer}
}

@inproceedings{sui2015safe,
  title={Safe exploration for optimization with {G}aussian processes},
  author={Sui, Yanan and Gotovos, Alkis and Burdick, Joel and Krause, Andreas},
  booktitle={International conf. on machine learning},
  pages={997--1005},
  year={2015},
  organization={PMLR}
}

@article{tsiamis2023statistical,
  title={Statistical learning theory for control: A finite-sample perspective},
  author={Tsiamis, Anastasios and Ziemann, Ingvar and Matni, Nikolai and Pappas, George J},
  journal=magazin,
  volume={43},
  number={6},
  pages={67--97},
  year={2023},
  publisher={IEEE}
}

@InProceedings{wagenmaker2020active,
  title = 	 {Active Learning for Identification of Linear Dynamical Systems},
  author =       {Wagenmaker, Andrew and Jamieson, Kevin},
  booktitle = 	 {Proceedings of Thirty Third Conf. on Learning Theory},
  pages = 	 {3487--3582},
  year = 	 {2020},
  volume = 	 {125},
  publisher =    {PMLR},
}

@inproceedings{ziemann2023tutorial,
  title={A tutorial on the non-asymptotic theory of system identification},
  author={Ziemann, Ingvar and Tsiamis, Anastasios and Lee, Bruce and Jedra, Yassir and Matni, Nikolai and Pappas, George J},
  booktitle=CDC,
  pages={8921--8939},
  year={2023},
  organization={IEEE}
}

@inproceedings{chowdhury2021noregret,
  title={No-regret algorithms for multi-task {B}ayesian optimization},
  author={Chowdhury, Sayak Ray and Gopalan, Aditya},
  booktitle={International Conf. on Artificial Intelligence and Statistics},
  pages={1873--1881},
  volome={130},
  year={2021},
  organization={PMLR}
}

@InProceedings{martinez2026vector,
  title = 	 {Vector-valued self-normalized concentration inequalities beyond sub-Gaussianity},
  author =       {Martinez-Taboada, Diego and Gonz{\'a}lez, Tom{\'a}s and Ramdas, Aaditya},
  booktitle = 	 {Proceedings of The 37th International Conf. on Algorithmic Learning Theory},
  pages = 	 {1--31},
  year = 	 {2026},
  editor = 	 {Telgarsky, Matus and Ullman, Jonathan},
  volume = 	 {313},
  series = 	 {Proceedings of Machine Learning Research},
  month = 	 {23--26 Feb},
  publisher =    {PMLR}
}

@article{campi2002finite,
  title={Finite sample properties of system identification methods},
  author={Campi, Marco C and Weyer, Erik},
  journal={IEEE Transactions on Automatic Control},
  volume={47},
  number={8},
  pages={1329--1334},
  year={2002},
  publisher={IEEE}
}

@article{whitehouse2026time,
  title={Time-uniform self-normalized concentration for vector-valued processes},
  author={Whitehouse, Justin and Wu, Zhiwei Steven and Ramdas, Aaditya},
  journal={The Annals of Applied Probability},
  volume={36},
  number={3},
  pages={1972--2013},
  year={2026},
  publisher={Institute of Mathematical Statistics}
}

@article{aolaritei2025stopping,
  title={Stopping rules for stochastic gradient descent via anytime-valid confidence sequences},
  author={Aolaritei, Liviu and Jordan, Michael I},
  journal={arXiv preprint arXiv:2512.13123},
  year={2025}
}
